\documentclass[12pt]{amsart}
\usepackage{amsfonts}
\usepackage{ifthen}
\usepackage{verbatim}
\usepackage{amsthm}
\usepackage{amsmath}
\usepackage{amssymb}
\usepackage{graphicx}
\usepackage{subfigure}
\usepackage{amscd,amssymb,amsthm}
\usepackage{color,xcolor}

\newcounter{minutes}
\divide\time by 60
\newcounter{hours}
\multiply\time by 60 \addtocounter{minutes}{-\time}

\title{On a convexity property }

\author{Slavko Simi\'c}

\address{ Mathematical Institute SANU, Kneza Mihaila 36, 11000
Belgrade, Serbia} \email{ ssimic@turing.mi.sanu.ac.rs}


\newtheorem{lemma}[equation]{Lemma}

\newtheorem{corollary}[equation]{Corollary}
\newtheorem{remark}[equation]{Remark}

\newcommand{\beq}{\begin{equation}}
\newcommand{\eeq}{\end{equation}}

\numberwithin{equation}{section}

\begin{document}

\def\thefootnote{}
\footnotetext{ \texttt{\tiny File:~\jobname .tex,
          printed: \number\year-\number\month-\number\day,
          \thehours.\ifnum\theminutes<10{0}\fi\theminutes}
} \makeatletter\def\thefootnote{\@arabic\c@footnote}\makeatother

\maketitle

\begin{abstract}
In this article we proved an interesting property of the class of
continuous convex functions. This leads to the form of
pre-Hermite-Hadamard inequality which in turn admits a
generalization of the famous Hermite-Hadamard inequality. Some
further discussion is also given.
\end{abstract}


\section{Introduction}

\vspace{0.5cm}

Most general class of convex functions is defined by the
inequality

\begin{equation}\label{1}
\frac{\phi(x)+\phi(y)}{2}\ge \phi(\frac{x+y}{2}).
\end{equation}

\vspace{0.5cm}

A function which satisfies this inequality in a certain closed
interval $I$ is called {\it convex} in that interval.
Geometrically it means that the midpoint of any chord of the curve
$y=\phi(x)$ lies above or on the curve.

\vspace{0.5cm}

Denote now by $Q$ the family of {\it weights} i.e., non-negative
real numbers summing to $1$. If $\phi$ is continuous, then the
inequality

\begin{equation}\label{2}
p\phi(x)+q\phi(y)\ge\phi(px+qy)
\end{equation}

holds for any $p,q\in Q$. Moreover, the equality sign takes place
only if $x=y$ or $\phi$ is linear (cf. \cite{hlp}).

\vspace{0.5cm}

The same is valid for so-called {\it Jensen functional}, defined
as

$$
\mathcal{J}_\phi(\bold{p},\bold{x}):=\sum p_i\phi(x_i)-\phi(\sum
p_ix_i),
$$

where $\bold{p}=\{p_i\}_1^n \in Q, \bold{x}=\{x_i\}_1^n\in I, n\ge
2$.

 \vspace{0.5cm}

 Geometrically, the inequality \eqref{2} asserts that each chord
 of the curve $y=\phi(x)$ lies above or on the curve.

 \vspace{0.5cm}

\section{Results and proofs}

\vspace{0.5cm}

Main contribution of this paper is the following

 \vspace{0.5cm}

{\bf Proposition X} {\it Let $f(\cdot)$ be a continuous convex
function defined on a closed interval $[a, b]:=I$. Denote
$$
F(s,t):=f(s)+f(t)-2f({s+t\over 2}).
$$

\vspace{0.5cm}

Prove that}
$$
\max_{s, t\in I} F(s,t)=F(a, b).\eqno (1)
$$

\vspace{0.5cm}

\begin{proof}

It suffices to prove that the inequality
$$
F(s,t)\le F(a, b)
$$
holds for $a<s<t<b$.

\vspace{0.5cm}

 In the sequel we need the following
assertion (which is of independent interest).

\vspace{0.5cm}

 \begin{lemma}  Let $f(\cdot)$ be a continuous convex
function on some interval $I\subseteq \mathbb R$. If $x_1, x_2,
x_3\in I$ and $x_1<x_2<x_3$, then

\vspace{0.5cm}

$$
(i) \ \ \ {f(x_2)-f(x_1)\over 2}\le f({x_2+x_3\over
2})-f({x_1+x_3\over 2});
$$
$$
(ii) \ \ \ {f(x_3)-f(x_2)\over 2}\ge f({x_1+x_3\over
2})-f({x_1+x_2\over 2}).
$$
\end{lemma}

\vspace{0.5cm}

{\bf Proof}

\vspace{0.5cm}

We shall prove the first part of the lemma; proof of the second
part goes along the same lines.

\vspace{0.5cm}

 Since  $x_1<x_2<{x_2+x_3\over 2}< x_3$, there exist
$p, q; \ 0\le p, q\le 1, p+q=1$ such that $x_2=px_1+q{x_2+x_3\over
2}$.

\vspace{0.5cm}

 Hence,
$$
{f(x_1)-f(x_2)\over 2} + f({x_2+x_3\over 2})={1\over
2}[f(x_1)-f(px_1+q{x_2+x_3\over 2})] + f({x_2+x_3\over 2})
$$
$$
\ge{1\over 2}[f(x_1)-(pf(x_1)+qf({x_2+x_3\over 2}))]+
f({x_2+x_3\over 2}) ={q\over 2} f(x_1)+{2-q\over 2}
f({x_2+x_3\over 2})
$$
$$
\ge f({q\over 2} x_1+{2-q\over 2}({x_2+x_3\over 2}))=f({q\over 2}
x_1+({x_2+x_3\over 2})-{1\over 2}(x_2-px_1))=f({x_1+x_3\over 2}).
$$

\vspace{0.5cm}

For the proof of second part we can take $x_2=p({x_1+x_2\over
2})+qx_3$ and proceed as above.

\vspace{0.5cm}

 Now, applying the part (i) with $x_1=a, x_2=s,
x_3=b$ and the part (ii) with $x_1=s, x_2=t, x_3=b$, we get
$$
{f(s)-f(a)\over 2}\le f({s+b\over 2})-f({a+b\over 2}); \eqno (2)
$$
$$
{f(b)-f(t)\over 2}\ge f({s+b\over 2})-f({s+t\over 2}), \eqno (3)
$$
respectively.

\vspace{0.5cm}

 Subtracting (2) from (3), the desired inequality follows.
\end{proof}

\vspace{0.5cm}

\begin{remark} A challenging task is to find a geometric proof of
the property $(1)$.
\end{remark}

\vspace{0.5cm}

We shall quote now a couple of important consequences. The first
one is used in a number of articles although we never saw a proof
of it.

\vspace{0.5cm}

\begin{corollary} Let $f$ be defined as above. If $x,y\in [a,b]$ and $x+y=a+b$, then

$$
f(x)+f(y)\le f(a)+f(b).
$$

\end{corollary}

\begin{proof} Obvious, as a simple application of Proposition X.

\end{proof}

\vspace{0.5cm}

\begin{corollary}  Under the conditions of Proposition X, the double inequality
$$
2f(\frac{a+b}{2})\le f(pa +qb)+f(pb+qa)\le f(a)+f(b)\eqno(4)
$$

holds for arbitrary weights $p,q\in Q$.
\end{corollary}

\vspace{0.5cm}

\begin{proof}

\vspace{0.5cm}

 Applying Proposition X with $s=pa+qb, t=pb+qa; s,t\in I$ we get the right-hand
side of (4). The left-hand side inequality is obvious since, by
definition,

$$
\frac{f(pa+qb)+f(pb+qa)}{2}\ge
f[\frac{(pa+qb)+(pb+qa)}{2}]=f(\frac{a+b}{2}).
$$

\end{proof}

\vspace{0.5cm}

\begin{remark} The relation (4) represents a kind of pre-Hermite-Hadamard
inequalities. Indeed, integrating both sides of (4) over $p\in
[0,1]$, we obtain the form of Hermite-Hadamard inequality (cf.
\cite{np}),

$$
f(\frac{a+b}{2})\le\frac{1}{b-a}\int_a^b
f(t)dt\le\frac{f(a)+f(b)}{2}.
$$
\end{remark}

\vspace{0.5cm}

Moreover, the inequality (4) admits a generalization of the
Hermite-Hadamard inequality.

\vspace{0.5cm}

{\bf Proposition Y} {\it Let $g$ be an arbitrary non-negative and
integrable function on $I$. Then, with $f$ defined as above, we
get

$$
2f(\frac{a+b}{2})\int_a^b g(t)dt\le\int_a^b
(g(t)+g(a+b-t))f(t)dt\le (f(a)+f(b))\int_a^b g(t)dt. \eqno(5)
$$}

\begin{proof} Multiplying both sides of $(4)$ with $g(pa+qb)$ and
integrating over $p\in [0,1]$, we obtain

$$
2f(\frac{a+b}{2})\frac{\int_a^b g(t)dt}{b-a}\le\frac{\int_a^b
(f(t)+f(a+b-t))g(t)dt}{b-a}\le (f(a)+f(b))\frac{\int_a^b
f(t)dt}{b-a},
$$

and, because

$$
\int_a^b (f(t)+f(a+b-t))g(t)dt=\int_a^b (g(t)+g(a+b-t))f(t)dt,
$$

the inequality $(5)$ follows.
\end{proof}

\vspace{0.5cm}

We shall give in the sequel some illustrations of this
proposition.

\vspace{0.5cm}

\begin{corollary} For any $f$ that is convex and continuous on $I:=[a,b], 0<a<b$ and $\alpha\in\mathbb
R/\{0\}$, we have

$$
2f(\frac{a+b}{2})\le\frac{\alpha}{b^\alpha-a^\alpha}\int_a^b
[t^{\alpha-1}+(a+b-t)^{\alpha-1}]f(t)dt\le f(a)+f(b).
$$
\end{corollary}

\vspace{0.5cm}

Also, for $\alpha\to 0$, we get

\begin{corollary}

$$
2f(\frac{a+b}{2})\frac{\log(b/a)}{a+b}\le\int_a^b
\frac{f(t)}{t(a+b-t)}dt\le [f(a)+f(b)]\frac{\log(b/a)}{a+b}.
$$

\end{corollary}

\vspace{0.5cm}

Similarly,

\begin{corollary}

$$
2f(\frac{\pi}{2})\le\int_0^\pi f(t)\sin t dt\le f(0)+f(\pi);
$$

$$
2f(\frac{\pi}{4})\le\int_0^{\pi/2} [\sin t+\cos t]f(t) dt\le
f(0)+f(\pi/2).
$$
\end{corollary}

\vspace{0.5cm}

Estimations of the convolution of symmetric kernel on a symmetric
interval are also of interest.

\vspace{0.5cm}

\begin{corollary} Let $f$ and $g$ be defined as above on a
symmetric interval $[-a,a], a>0$. Then we have that

$$
2f(0)\int_{-a}^a g(t) dt\le\int_{-a}^a [g(-t)+g(t)]f(t)dt\le
[f(-a)+f(a)]\int_{-a}^a g(t)dt.
$$

\end{corollary}

\vspace{0.5cm}

\begin{remark} There remains the question of possible extensions
of the relation $(1)$. In this sense one can try to prove, along
the lines of the proof of $(1)$, that

$$
\max_{p,q\in Q; x,y\in [a,b]}F^*(p,q;x,y)=F^*(p,q;a,b),
$$

where

$$
F^*(p,q;x,y):=pf(x)+qf(y)-f(px+qy).
$$

\vspace{0.5cm}

Anyway the result will be wrong, as simple examples show (apart
from the case $f(x)=x^2$).

On the other hand, it was proved in {\cite s} that for $p_i\in Q$
and $x_i\in [a,b]$ there exist $p,q\in Q$ such that

$$
\mathcal{J}_f(\bold{p},\bold{x})=\sum p_i f(x_i)-f(\sum p_i
x_i)\le pf(a)+qf(b)-f(pa+qb),\eqno(6)
$$

for any continuous function $f$ which is convex on $[a,b]$.

\vspace{0.5cm}

Therefore, an important conclusion follows.

\end{remark}

\vspace{0.5cm}

\begin{corollary} For arbitrary $p_i\in Q$
and $x_i\in [a,b]$, we have that

$$
\sum p_i f(x_i)-f(\sum p_i x_i)\le \max_p
[pf(a)+qf(b)-f(pa+qb)]:=T_f(a,b),
$$

where $T_f(a,b)$ is an optimal upper global bound, depending only
on $a$ and $b$ (cf. \cite{s}).

\end{corollary}

\vspace{0.5cm}

An answer to the above remark is given by the next

\vspace{0.5cm}

{\bf Proposition Z} {\it If $f$ is continuous and convex on
$[a,b]$, then

$$
\max_{p,q\in Q; x,y\in [a,b]}F^*(p,q;x,y)\le F(a,b).
$$}

\begin{proof} We shall prove just that

$$
F^*(p,q;x,y)\le F(x,y),
$$

for all $p,q\in Q$ and $x,y\in [a,b]$.

\vspace{0.5cm}

Indeed,

$$
F(x,y)-F^*(p,q;x,y)=qf(x)+pf(y)+f(px+qy)-2f(\frac{x+y}{2})
$$
$$
 \ge f(qx+py)+f(px+qy)-2f(\frac{x+y}{2})\ge
2f(\frac{(qx+py)+(px+qy)}{2})-2f(\frac{x+y}{2})=0.
$$

\vspace{0.5cm}

The rest of the proof is an application of Proposition X.
\end{proof}

\vspace{0.5cm}

Putting there $x=a, y=b$ and combining with $(6)$, we obtain
another global bound for Jensen functional.

\vspace{0.5cm}

\begin{corollary} We have that

$$
\mathcal{J}_f(\bold{p},\bold{x})\le
f(a)+f(b)-2f(\frac{a+b}{2}):=T'_f(a,b).
$$

\end{corollary}

\vspace{0.5cm}

The bound $T'_f(a,b)$ is not so precise as $T_f(a,b)$ but is much
easier to calculate.

\end{document}